\RequirePackage{fix-cm}
\documentclass[reqno,twoside]{amsart}

\usepackage{amsfonts}
\usepackage{amsmath,amssymb,amsthm}
\usepackage{mathtools}
\usepackage{etoolbox}
\usepackage{color}
\usepackage[english]{babel}
\usepackage[colorlinks=true,citecolor=red,linkcolor=blue,pdfpagetransition=Blinds]{hyperref}
\usepackage{microtype}
\usepackage{fullpage}
\usepackage{floatrow}
\usepackage[utf8]{inputenc}
\usepackage[T1]{fontenc}
\usepackage{lmodern}
\usepackage[leftcaption]{sidecap}
\sidecaptionvpos{figure}{m}
\usepackage{graphicx}
\usepackage{lipsum}
\usepackage[mathcal]{eucal}
\usepackage{tikz}
\usepackage{pgfplots}
\usepackage{pgfplotstable}
\usepackage{mathrsfs}
\usepackage{stmaryrd}
\usepackage{subfig}

\pgfplotsset{surface/.style={ %
		view={130}{20},
               axis z line=middle,%
               axis x line=middle,%
               axis y line=middle,%
               xmax=1.2,%
               ymax=1.2,%
               clip=false,
               xtick=\empty,
               ytick=\empty,
               extra x ticks={1},%
               extra y ticks={1},%
               extra x tick labels={$T=1$},%
               zmajorticks=true,
               colormap/jet}
}

\pgfplotsset{erreurs/.style={legend cell align=left,
    legend pos=outer north east,
    legend plot pos=left,
    legend style={cells={anchor=west},draw=none},
    xlabel=$h$,
    xmin=0.001,
    xmax=0.03}
    }

\tikzset{pente/.style={opacity=0.6}}

\pgfplotsset{shadow/.style={black,mark=square,mark size=3.0,mark options={solid,fill=none}}}
\pgfplotsset{fast/.style={black,mark=triangle*,mark size=3.0,mark options={solid,fill=none}}}

\pgfplotsset{cout/.style={purple,mark=diamond*,mark size=3.2,mark options={fill=purple}}}
\pgfplotsset{cible/.style={blue,mark=square*,mark size=2.5,mark options={fill=blue}}}
\pgfplotsset{cibleyT/.style={black,mark=*,mark size=2.5,mark options={fill=gray}}}
\pgfplotsset{CG/.style={black,mark=otimes*,mark size=2.5,mark options={fill=gray}}}
\pgfplotsset{solex/.style={black,mark=*,mark size=2.5,mark options={fill=gray}}}
\pgfplotsset{energie/.style={red,mark=*,mark size=2.5,mark options={fill=red}}}

\usetikzlibrary{matrix,external}
\usepackage{mwe}

\newtheorem{theorem}{Theorem}[section]

\newtheorem{lemma}[theorem]{Lemma}
\newtheorem{proposition}[theorem]{Proposition}

\numberwithin{equation}{section}
\numberwithin{figure}{section}
\numberwithin{theorem}{section}

\DisableLigatures{encoding = *, family = * }

\newcommand{\norm}[2]{\left\|#1\right\|_{#2}}

\newcommand{\RR}{\mathbb{R}} 

\newcommand\dis{\displaystyle}

\newcommand\inter[1]{\llbracket #1\rrbracket}

\def\Xint#1{\mathchoice
{\XXint\displaystyle\textstyle{#1}}%
{\XXint\textstyle\scriptstyle{#1}}%
{\XXint\scriptstyle\scriptscriptstyle{#1}}%
{\XXint\scriptscriptstyle\scriptscriptstyle{#1}}%
\!\int}
\def\XXint#1#2#3{{\setbox0=\hbox{$#1{#2#3}{\int}$ }
\vcenter{\hbox{$#2#3$ }}\kern-.6\wd0}}

\def\fint{\Xint-}

\DisableLigatures{encoding = *, family = * }

\title[Optimal control non-local parabolic problems]{Optimal control of linear non-local parabolic problems with an integral kernel}

\author{Umberto Biccari\textsuperscript{\,$\ast$}}  
\address{\textsuperscript{$\ast$}\, [1] Chair of Computational Mathematics, Fundaci\'on Deusto, Avenida de las Universidades 24, 48007 Bilbao, Basque Country, Spain} 
\address{[2]\,Facultad de Ingenier\'ia, Universidad de Deusto, Avenida de las Universidades 24, 48007 Bilbao, Basque Country, Spain.}
\email{umberto.biccari@deusto.es, u.biccari@gmail.com}

\thanks{This project has received funding from the European Research Council (ERC) under the European Union’s Horizon 2020 research and innovation programme (grant agreement NO: 694126-DyCon). The work of the U.B. is partially supported by the Air Force Office of Scientific Research (AFOSR) under Award NO: FA9550-18-1-0242, by Grant MTM2017-92996-C2-1-R COSNET of MINECO (Spain) and by the Elkartek grant KK-2020/00091 CONVADP of the Basque government. The work of V. H.-S. is partially supported by the program ``Estancias posdoctorales por M\'exico'' of CONACyT, Mexico. A.~O. would like to thank all members of the Chair of Computational Mathematics at Fundaci\'on Deusto, Bilbao, Spain for their kind hospitality during his visit, which was useful for developing part of this paper.} 

\author{V. Hern\'andez-Santamar\'ia\textsuperscript{\,$\dagger$}}  
\address{\textsuperscript{$\dagger$}\,Instituto de Matem\'aticas, Universidad Nacional Aut\'onoma de M\'exico, Circuito Exterior, C.U., C.P. 04510 CDMX, Mexico}
\email{victor.santamaria@im.unam.mx}

\author{L. Louison\textsuperscript{\,$\ddagger$}}
\address{\textsuperscript{$\ddagger$}\, Universit{\'e} de Guyane, UMR 42 EcoFoG, Route de Montabo BP 165, 97323 Cayenne ({\sc French Guiana})}
\email{loic.louison@univ-guyane.fr}      

\author{A. Omrane\textsuperscript{\,$\ddagger$\,$\bullet$}}
\address{\textsuperscript{$\ddagger$}\, Universit{\'e} de Guyane, UMR 42 EcoFoG, Route de Montabo BP 165, 97323 Cayenne ({\sc French Guiana})}
\address{\textsuperscript{$\bullet$}\, Universit\'e de Bordeaux, IMB, 351 Cours de la Lib\'eration, 33400 Talence (France)}
\email{abdennebi.omrane@univ-guyane.fr}


\begin{document}

\begin{abstract} 
We consider a linear non-local heat equation in a bounded domain $\Omega\subset\mathbb{R}^d$, $d\geq 1$, with Dirichlet boundary conditions, where the non-locality is given by the presence of an integral kernel. Motivated by several applications in biological systems, in the present paper we study some optimal control problems from a theoretical and numerical point of view. In particular, we will employ the classical low-regret approach of J.-L. Lions for treating the problem of incomplete data and provide a simple computational implementation of the method. The effectiveness of the results are illustrated by several examples. 
\end{abstract}

\maketitle 

\section{Introduction}\label{intro_sec} 

Non-local models are used for the description of several complex phenomena for which a local approach is inappropriate or limiting. In the past, they have been successfully employed in different fields, such as material sciences (\cite{bates2006some}), dislocation dynamics in crystals (\cite{dipierro2015dislocation}), image processing (\cite{gilboa2008nonlocal}) or elasticity (\cite{silling2000reformulation}).

Non-local terms appear also in many diffusion models in biology, where they describe biological systems on scales that are convenient to observation and data collection. A classical example is the chemotaxis phenomena, which can be modeled by non-local differential  equations of diffusion type. For instance, in \cite{burger2006keller} a non-local Keller-Segel model is used to study the chemotaxis of a cell population in which the chemicals are produced by the cells themselves. Furthermore, one can find numerous examples of other non-local diffusion phenomena in the book \cite{okubo2013diffusion}, where they are presented several systems involving, for instance, random walks and random population diffusion. 

Motivated also by the aforementioned applications, these types of problems have recently raised the interest of the mathematical community. In particular, we may refer to the monograph \cite{andreu2010nonlocal} for an extended mathematical survey containing different examples of non-local diffusion models. 

In this paper, we consider a diffusion spatially non-local equation, where the non-locality is given by the presence of an integral kernel. In particular, we are interested in the analysis of control properties. 

Control Theory deals with dynamical systems that can be controlled, i.e. whose evolution can be influenced by some external agent. In this framework, a first problem one could face is the one of \textit{controllability}, that is, to study the possibility of steering the system from one configuration to another by employing one or more controls applied through actuators. If controllability to a given final state is granted, one can try to reach this state by minimizing
some cost, thus defining an \textit{optimal control} problem. The mathematical theory of optimal control has rapidly developed in the last decades into an important and separate field of applied mathematics. Actually, applications of optimal control of parabolic equations span a large spectrum of fields including aviation and space technology, robotics, movement sequences in sports, and the control of chemical processes.

For the non-local diffusion models that we are going to consider in the present work, which will be introduced in the next section, controllability has been already treated in some recent paper (see, for instance, \cite{biccari2017null,fernandez2016null,lissy2018internal,micu2017local}). On the other hand, to the best of our knowledge, the analysis of optimal control problems was never addressed in this setting. Then, the main purpose of our work will be to give an insight on some classical optimal control question which arises naturally from the biological applications behind the mentioned models. In this framework, we will focus our attention on two different situations:
\begin{enumerate}
	\item In a first moment, we will deal with a classical optimal control problem, in which we are interested in minimizing a given cost functional coming from the controllability theory. This cost functional will typically characterize the optimal control of minimal energy which allows to reach some predetermined configuration.
	\item Secondly, we will consider the case of missing or incomplete initial data. This is a very typical situation when studying certain natural phenomena, for instance in ecology or genetics, where the employment of models with no missing data may lead to a misunderstanding of the underneath processes. 
\end{enumerate} 

This second problem will be addressed by using the concepts of \textit{low-regret} or \textit{least-regret} controls, which have been introduced by J.-L. Lions in \cite{lions1992controle}, and may be well-adapted to the non-local problem we are considering. 

These techniques were designed precisely for dealing with the optimal control of models where missing data are incorporated. They consist of transferring the optimal control problem with missing data to a classical one, relaxing the problem of control by a sequence of low-regret controls. This idea was firstly applied to linear models (\cite{gabay1994decisions,lions1999duality}), and it was later extended to the nonlinear setting in \cite{nakoulima2002no}. Afterwards, many other authors applied these ideas to other control problems for PDEs with incomplete data. The interested reader may refer, for instance, to \cite{dorville2004low,jacob2010optimal,louison_1,mahoui2017pointwise}. In this work, we will see how the low-regret and least-regret approaches can be applied also in the the non-local context.

The present paper is organized as follows: in Section \ref{sec:2}, we present the problem we are going to consider. In Section \ref{sec:3}, we address some elementary optimal control question associated with our model, while in Section \ref{sec:4}, we address the issue of missing initial data. Finally, Section \ref{sec:5} is devoted to some numerical experiment. 

\section{Problem formulation and existing controllability results}\label{sec:2} 

Let $\Omega$ be a bounded domain of $\RR^d$, $d\geq 1$, with boundary $\partial\Omega$ of class $\mathcal{C}^2$. Given a time horizon $T>0$, we set $Q:=\Omega\times (0,T)$ and $\Sigma:=\partial\Omega\times(0,T)$. Let $K(x,\theta,t)\in L^\infty(\Omega\times\Omega\times(0,T))$, and  consider the following linear non-local parabolic problem
\begin{align}\label{direct_pb_y}
	\begin{cases}
		\dis\frac{\partial y}{\partial t}-\Delta y+ \int_\Omega K(x,\theta,t)\,y(\theta,t)\,d\theta = v{\bf 1}_{\mathcal{O}}, & (x,t)\in Q,
		\\
		y = 0, &(x,t)\in \Sigma,
		\\
		y(0,x) = y_0(x), & x\in\Omega.
	\end{cases}
\end{align}

In \eqref{direct_pb_y}, $y:=y(x,t)$ is the state function and $v:=v(x,t)$ is the control function. The latter acts on the system through the non-empty set $\mathcal{O}\in\Omega$. Here, ${\bf 1}_{\mathcal{O}}$ denotes the characteristic function of $\mathcal{O}$. 

Moreover, we assume that $y_0\in L^2(\Omega)$ and $v\in \mathcal{U}:=L^2(\mathcal{O}\times(0,T))$, so that the system \eqref{direct_pb_y} admits a unique solution
\begin{align*}
	y\in L^2((0,T); H^1_0(\Omega))\cap H^1(0,T; H^{-1}(\Omega)),
\end{align*}
which satisfies classical energy estimates. Actually, this remains true also if $v{\bf 1}_{\mathcal{O}}$ is replaced by a more general right-hand side $f\in L^2(0,T; H^{-1}(\Omega))$.

Non-local diffusion equations in the form \eqref{direct_pb_y} were introduced as an extension of already existing advection-diffusion-reaction models of multi-species ecosystems. They take into account the fact that, in many biological situations, local movement is also coupled with long-range influences, such as the combination of clonal growth and a dispersing phase like seeds. One example of a model for this situation was developed by Furter and Grinfeld (\cite{furter1989local}), who examined diffusion-reaction models of single-species dynamics that incorporate a reaction term dependent on characteristics of the population as a whole.

Concerning the controllability of \eqref{direct_pb_y}, this issue has been firstly studied in \cite{fernandez2016null}, where the authors proved an interior null controllability result by imposing analyticity assumptions on the kernel in order to obtain unique continuation properties. In this framework, also coupled systems have been treated in \cite{lissy2018internal}. Moreover, in \cite{micu2017local}, analogous results have been obtained for a one-dimensional equation with a kernel in separated variables, by means of classical spectral analysis techniques. Later on, these mentioned results have been extended in \cite{biccari2017null}, where the null-controllability is proved under weaker assumptions on the kernel and for both the linear and the semilinear case, by using a Carleman approach. In particular, the authors there proved the following result.

\begin{theorem}[{\cite[Theorem 1.1]{biccari2017null}}]\label{theo1_bhs}
Let $T>0$ and assume that $K\in L^\infty(\Omega\times\Omega\times(0,T))$ satisfies 
\begin{align}\label{K_est_weak}
	\mathcal{K}:=\sup_{(x,t)\in\overline{Q}}\exp\left(\frac{\mathcal{A}}{t(T-t)}\right)\int_{\Omega} |K(x,\theta,t)|\,d\theta <+\infty,
\end{align}
for a given positive constant $\mathcal A$ only depending on $\Omega$. Then, for any $y_0\in L^2(\Omega)$, there exists a control function $v\in \mathcal U$ such that the associated solution $y$ of \eqref{direct_pb_y} satisfies $y(x,T) = 0$.	
\end{theorem}	

We shall stress that, according to the hypothesis \eqref{K_est_weak}, to get a positive controllability result the kernel $K$, as a function of $t$, should behave like 
\begin{align*}
	K(\cdot,\cdot,t) \sim \exp\left(-\frac{\mathcal B}{t(T-t)}\right),
\end{align*}
i.e. it should decay exponentially as $t$ goes to $0$ and $T$. 

This is however a quite strong restriction on the admissible kernels. Actually, only rapidly decaying or compactly supported (in time) kernels verify \eqref{K_est_weak}.

In some specific situations, this assumption it is actually not necessary. For instance, in the particular case when $K=|\Omega|^{-1}\textnormal{cst.}$, assumption \eqref{K_est_weak} has been recently removed in \cite{hernandez2020local}. Using the so-called shadow model (see \cite{HR00,MCHKS18} or \cite{HSZ20} for other applications in control), the authors have proved that the system 
\begin{equation}\label{eq:nonlocal_nonlinear}
	\frac{\partial y}{\partial t}-\Delta y + a y + b\,  \fint_{\Omega} y(\theta,t)d\theta = v \mathbf{1}_{\mathcal O}
\end{equation}
where $a,b\in\mathbb R$ and 
\begin{align*}
	\fint_\Omega y(\theta,t)d\theta:=|\Omega|^{-1}\int_{\Omega}y(\theta,t)d\theta,
\end{align*}
is null-controllable at time $T>0$. Obviously, since the kernel $K$ is constant, the controllability result in Theorem \ref{theo1_bhs} does not apply directly. We also emphasize that this result extends the one in \cite{fernandez2016null} since the constant kernels do not satisfy the analyticity assumptions considered in that work. 

While the controllability of \eqref{direct_pb_y} has been already studied in the aforementioned references, to the best of our knowledge the optimal control for this same problem is yet to be addressed. This is exactly the purpose of the present paper, in which this issue will be treated from two different viewpoints. We will firstly address a classical optimal control problem, in which we are interested in minimizing a given functional characterizing the control $v$ employed in \eqref{direct_pb_y}. In more detail, we consider the strictly convex cost functional $J:\mathcal U\to\RR$ defined by
\begin{equation}\label{quadratic_cost}
	J(v)={\beta}\norm{y(v)-\bar{z}}{L^2(Q)}^2 +\mu\norm{v}{\mathcal{U}}^2, \qquad \beta,\,\mu>0,
\end{equation}
where we denote by $y=y(v)$ the dependence of the state in terms of the control function $v$, and we will study the optimal control problem 
\begin{equation}\label{min_J_part1}
	\inf_{v\in\mathcal{U}} J(v).
\end{equation}

To this end, we will firstly prove the existence of a minimizer and, in a second moment, we will characterize it through a suitable optimality system.

In \eqref{quadratic_cost}, $\bar{z}\in L^2(\Omega\times(0,T))$ is the target we aim to reach and $\mathcal U$ is a given set of admissible controls. Here we will consider the simple case in which $\mathcal U = L^2(\mathcal O\times(0,T))$, although other different choices are possible. Actually, in many practical situations realistic controls depend on the physical process that the equation is describing, and the set of admissible controls may be characterized by taking into account several constraints.


\section{Optimal control}\label{sec:3}

As we mentioned in Section \ref{intro_sec}, the mathematical theory of optimal control is nowadays very rich in the context of local equations, covering for instance linear and non-linear models, convex and non-convex functionals, and problems with control constraints. On the other hand, way less results are available in the framework of non-local models. Actually, to the best of our knowledge, the optimal control problem for \eqref{direct_pb_y} has never been addressed before. 

In this section, we will study the optimal control problem \eqref{min_J_part1}. In more detail, we will firstly prove the existence of a minimizer for the functional $J$ and, in a second moment, we will give a characterization through the optimality system. Moreover, by means of classical techniques in control theory, in what follows we will need to rely on the so-called \textit{adjoint equation}, which is given by 
\begin{align}\label{adjoint_pb_p}
	\begin{cases}
		\displaystyle -\frac{\partial p}{\partial t}-\Delta p+\int_\Omega K(\theta,x,t)\,p(\theta,t)\,d\theta = y(v)-\bar{z}, & (x,t)\in Q,
		\\
		p = 0, & (x,t)\in \Sigma,
		\\
		p(x,T) = 0, & x\in\Omega.
	\end{cases}
\end{align}

\subsection{Existence of an optimal control}

We have the following.
\begin{proposition}\label{exist_prop_sect1} 
There exists an optimal control function $u\in\mathcal{U}$, unique solution to the minimization problem \eqref{quadratic_cost}, \eqref{min_J_part1}.
\end{proposition}

\begin{proof}
The result is a consequence of the so-called \textit{Direct Method of the Calculus of Variations} (see \cite{brezis2010functional}), which ensures the existence and uniqueness of a minimizer provided the functional $J$ satisfies the following assumptions:
\begin{itemize}
	\item[1.] $J$ is lower semi-continuous.
	\item[2.] $J$ is strictly convex, i.e. $J((1-\lambda)v+\lambda w) < (1-\lambda)J(v)+\lambda J(w)$ for all $\lambda\in (0,1)$ and $v,w\in\mathcal U$.
	\item[3.] $J$ is coercive, i.e. $\lim_{\norm{v}{\mathcal U}\to+\infty}J(v) = +\infty$.	
\end{itemize}

The first two properties are evident since the functional is quadratic. Concerning the third one, it is enough to see that $J(v)\geq \mu\norm{v}{\mathcal U}$. This ends the proof.  
 \end{proof}

\subsection{Characterization of the optimal control}

Notice that Proposition \ref{exist_prop_sect1} guarantees the existence of a unique solution to our optimal control problem \eqref{quadratic_cost}-\eqref{min_J_part1}, but it does not allow to characterize it. Notwithstanding, in practical applications it is important also an explicit knowledge of the control, which is typically computed through the so-called \textit{optimality system}. In order to introduce this system, we will firstly need the following technical lemma.

\begin{lemma} 
Let $v\in\mathcal U$ be the optimal control minimizing the functional $J$. For all $w \in \mathcal{U}$, we have the identity
\begin{equation}\label{euler_lagrange_part1}
	{\beta}\langle(y(v)-\bar{z}),y(w)\rangle_{L^2(Q)} + \mu\langle v,w\rangle_{\mathcal{U}} = 0.
\end{equation}
\end{lemma} 

\begin{proof}
We use the necessary condition of Euler-Lagrange satisfied by the optimal control $v\in \mathcal{U}$. Thanks to the linearity of the problem with respect to the control function, for every $\lambda>0$ and $w\in \mathcal{U}$ we have $y(v+\lambda w)=y(v)+\lambda y(w)$. Then,
\begin{align*}
	J(v+\lambda w)-J(v) &= {\beta}\norm{y(v+\lambda w)-\bar{z}}{L^2(Q)}^2 + \mu\norm{v+\lambda w}{\mathcal{U}}^2 - {\beta}\norm{y(v)-\bar{z}}{L^2(Q)}^2 + \mu\norm{v}{\mathcal{U}}^2
	\\
	&={\beta\lambda^2}\norm{y(w)}{L^2(Q)}^2 + 2\beta\lambda\langle y(v)-\bar{z},y(w)\rangle_{L^2(Q)} + \mu\lambda^2\norm{w}{\mathcal{U}}^2 + 2\mu\lambda\langle v,w\rangle_{\mathcal{U}}.
\end{align*}
Hence, the necessary condition of Euler-Lagrange gives
\begin{align*}
	\lim_{\lambda\rightarrow 0} \left(\frac{J(v+\lambda w)-J(v)}{\lambda}\right) ={\beta} \langle y(v)-\bar{z}, y(w)\rangle_{L^2(Q)} + \mu\langle v,w\rangle_{\mathcal{U}} = 0, \; \mbox{ for all } w\in\mathcal{U}.
\end{align*}
\end{proof}

\noindent We can now give the following characterization of the optimal control.

\begin{proposition}\label{prop_characterization_part1}
The optimal control $v$ is characterized by the triplet 
\begin{align*}
	(v,y,p)\in \mathcal{U}\times L^2(Q)\times L^2(Q),
\end{align*}
unique solution to the optimality system
\begin{equation}\label{sos_part1}
	\begin{cases}
		\displaystyle \frac{\partial y}{\partial t}-\Delta y+\int_\Omega K(x,\theta,t)\,y(\theta,t)\,d\theta = v{\bf 1}_{\mathcal{O}}, & (x,t)\in Q,
		\\[9pt]
		\displaystyle -\frac{\partial p}{\partial t}-\Delta p+\int_\Omega K(\theta,x,t)\,p(\theta,t)\,d\theta = \beta(y(v)-\bar{z}), & (x,t)\in Q,
		\\[9pt]
		y=p=0, & (x,t)\in \Sigma,
		\\[9pt]
		y(x,0)=y_0(x), \quad p(x,T) = 0, & x\in\Omega,
	\end{cases}
\end{equation}
with
\begin{equation}\label{adjoint_state_eq_part1}
	 p+\mu v=0, \qquad \mbox{in}\quad \mathcal{O}\times (0,T).
\end{equation}
\end{proposition}

\begin{proof}
Without losing generality we can assume here that $y_0=0$, since in the case of non-zero initial data the result may be equivalently obtained by means of a simple change of variables. 

We multiply \eqref{direct_pb_y} by $p$ solution of \eqref{adjoint_pb_p} and we integrate over the domain $Q$. By taking into account the boundary and initial conditions, we get by Fubini's theorem
\begin{align*}
	\int_Q {\beta} (y(v)-\bar{z})y(w)\,dxdt &=\int_Q\left(-\frac{\partial p}{\partial t}-\Delta p+ \int_\Omega K(\theta,x,t)p(\theta,t)\,d\theta\right)y(x,t;w)\,dxdt
	\\
	&=\int_Q\left(\frac{\partial y}{\partial t}-\Delta y+\int_\Omega K(x,\theta,t)y(\theta,t;w)\,d\theta\right)p(x,t)\,dxdt 
	\\
	&= \int_0^T\int_{\mathcal{O}}pw\,dxdt.
\end{align*}
Then, from \eqref{euler_lagrange_part1} we immediately obtain $\langle\, p+\mu v,\, w\rangle_{\mathcal{U}}=0$ for all $w\in \mathcal{U}$,
from which we get \eqref{adjoint_state_eq_part1}.
\end{proof}

\section{Incomplete data problems}\label{sec:4}

In this section, we consider the case of a missing or incomplete initial datum $y_0$, which is a very typical situation when studying natural phenomena governed by \eqref{direct_pb_y}. This problem is described below and will be addressed by using the concept of no-regret controls, which has been introduced by J.-L. Lions in \cite{lions1992controle}.

\subsection{Preliminaries}

We are interested in studying the control of the model \eqref{direct_pb_y}-\eqref{quadratic_cost} when it describes the dynamics of the density population $y(x,t)$ in the case in which the initial data is missing or incomplete. 

This is a very typical situation when modeling many phenomena in physics, ecology, dynamic population, and several other fields. As a matter of fact, in those frameworks we may face the problem of incomplete data because of their practical inaccessibility, or because sometimes we have a great variety of possibilities when choosing for instance the initial conditions. In addition, boundary conditions may also be unknown or only partially known on a subset of the boundary that may be inaccessible to measurements. The same goes for source terms that can be difficult to access, or the structure of the domain, which can also be imperfectly known (for example in oil well management). 

In our case, we will focus on a problem with missing initial datum. This means that this time we are assuming that $y_0$ is an unknown function, belonging to some vector closed subspace $G$ of $L^2(\Omega)$. We are still concerned with optimal controls $v\in\mathcal{U}$, i.e. in solving
\begin{align*}
	\inf_{v\in {\mathcal U}}J(v,y_0),\qquad \mbox{ for all } y_0\in G.
\end{align*}
where $J(v,y_0)$ denotes the explicit dependence on $y_0$ of the functional $J$ (see Lemma \ref{lem:equiv_rest} for a precise description). Nevertheless, this minimization problem has actually no sense since $G$ is either the empty space or it has an infinite number of elements. 

To deal with this problem, we use the low-regret concept of J.-L. Lions (see \cite{lions1992controle}), which is well suited for incomplete data problems, and it is based on replacing \eqref{min_J_part1} by
\begin{equation}\label{min_J_part2}
	 \inf_{v\in {\mathcal U}}\Bigg(\sup_{y_0\in G} \Big(J(v,y_0)-J(0,y_0)-\gamma\|y_0\|_G^2\Big)\Bigg),
\end{equation}
where $\gamma$ is a small parameter.

The meaning of \eqref{min_J_part2} is to look for the control not making things worse with respect to doing nothing (i.e. case $v=0$). A solution to \eqref{min_J_part2} is called a low-regret optimal control (see \cite{jacob2010optimal,lions1992controle,louison_1,nakoulima2000perturbations} for further information on the method).

We show two preliminary results.

\begin{lemma}\label{lem:equiv_rest}
For every $v\in\mathcal{U}$ and for any $y_0\in G$, we have the property
\begin{equation}\label{form0_part2}
	J(v,y_0)-J(0,y_0)= J(v,0)-J(0,0)+ \beta\langle\,y(v,0),y(0,y_0)\rangle_{L^2(Q)}.
\end{equation}
\end{lemma}
\begin{proof}
We denote by $y(v,y_0):=y(x,t,v,y_0)$ the solution to \eqref{direct_pb_y} with control $v\in \mathcal U$ and initial datum $y_0\in G$. From the linearity of the problem \eqref{direct_pb_y} we have $y(v,y_0)=y(v,0)+y(0,y_0)$. Then
\begin{align*}
	J(v,y_0)-J(0,y_0) &= {\beta}\norm{y(v,y_0)-\bar{z}}{L^2(Q)}^2 + \mu\norm{v}{\mathcal{U}}^2-{\beta}\norm{y(0,y_0)-\bar{z}}{L^2(Q)}^2
	\\
	&= {\beta}\norm{y(v,0)}{L^2(Q)}^2 + 2\beta\langle y(v,0),y(0,y_0)-\bar{z}\rangle_{L^2(Q)} + \mu\norm{v}{\mathcal{U}}^2
	\\
	&= {\beta}\norm{y(v,0)-\bar{z}+\bar{z}}{L^2(Q)}^2 + 2\beta\langle y(v,0),y(0,y_0)-\bar{z}\rangle_{L^2(Q)} +\mu\norm{v}{\mathcal{U}}^2
	\\
	&= J(v,0)-J(0,0)+ 2\beta\langle y(v,0),y(0,y_0)\rangle_{L^2(Q)}.
\end{align*}
\end{proof}

\begin{lemma}\label{lem:adjoint} 
Let $\xi:=\xi(x,t,v,0)$ be the solution of the adjoint problem
\begin{equation}\label{adjoint_pb_xi}
	\begin{cases}
		\displaystyle -\frac{\partial \xi}{\partial t}-\Delta \xi+\int_\Omega K(\theta,x,t)\,\xi(\theta,t)\,d\theta = y(v,0), & (x,t)\in Q,
		\\
		\xi = 0, & (x,t)\in \Sigma,
		\\
		\xi(x,T) = 0, & x\in\Omega.
	\end{cases}
\end{equation}
Then, we have
\begin{equation}\label{form_part2}
	J(v,y_0)-J(0,y_0)= J(v,0)-J(0,0)+2\beta\langle \xi(0)\,,\,y_0\rangle_{G',G}
\end{equation}
where $\xi(0):=\xi(x,0,v,0)$ is the solution of the adjoint problem (\ref{adjoint_pb_xi}) at time $t=0$ and where $G'$ is the topological dual of $G$.
\end{lemma}
\begin{proof}
If we multiply the first equation of \eqref{adjoint_pb_xi} by $y(0,y_0)$ and we integrate by parts, we obtain
\begin{align*}
	\langle y(v,0),y(0,y_0)\rangle_{L^2(Q)} &= \int_Q y(v,0)y(0,y_0)\,dxdt
	\\
	&=\int_Q\left(-\frac{\partial \xi}{\partial t}-\Delta \xi+ \int_\Omega K(\theta,x,t)\,\xi(\theta,t)\,d\theta\right)y(0,y_0)\,dxdt
	\\
	&=\int_Q\left(\frac{\partial y}{\partial t}-\Delta y+\int_\Omega K(x,\theta,t)\,y(\theta,t,w)\,d\theta\right)\xi(x,t)\,dxdt + \displaystyle\int_\Omega \xi(0)y_0(x)\,dx
	\\
	&= \int_\Omega \xi(0)y_0(x)\,dx.
\end{align*}
Then, the result follows immediately by applying \eqref{form0_part2}.
\end{proof}

Lemma \ref{lem:adjoint} can now be employed to transform the inf/sup problem \eqref{min_J_part2} into a classical minimization one. Indeed, using \eqref{form_part2}, we obtain from \eqref{min_J_part2} that
\begin{align*}
	\displaystyle\inf_{v\in {\mathcal U} }\left[J(v,0)-J(0,0) + 2\beta \displaystyle\sup_{y_0\in G} \Big(\langle \xi(0),y_0\rangle_{G',G} - \frac{\gamma}{2\beta}\|y_0\|_{G}^2\Big)\right].
\end{align*}
This, together with the fact that
\begin{align*}
	\sup_{y_0\in G} \Big(\langle \xi(0),y_0\rangle_{G',G} -\frac{\gamma}{2\beta}\|y_0\|_{G}^2\Big) =\frac{\beta}{2\gamma}\|\xi(0)\|^2_{G'},
\end{align*}
gives us the following minimization problem equivalent to \eqref{min_J_part2}
\begin{align}\label{infgamma}
	&\inf_{v\in {\mathcal U} }{\mathcal{J}}^{\gamma}(v) 
	\\
	&\mathcal{J}^{\gamma}(v)=J(v,0)-J(0,0) + \frac{\beta^2}{\gamma}\Big\|\xi(0)\Big\|_{G'}^2. \notag
\end{align}
In what follows, we will always focus on \eqref{infgamma} instead of the original low-regret problem \eqref{min_J_part2}.

\subsection{Existence of the optimal low-regret control}

We discuss here the existence of a low-regret control for \eqref{direct_pb_y}. In particular, the main result of this section will be the following. 

\begin{theorem} 
There exists a unique low-regret optimal control function denoted by $v_\gamma \in{\mathcal U}$, solution to the minimization problem \eqref{infgamma}.
\end{theorem}

\begin{proof} 
First of all, notice that for all $v\in {\mathcal U}$ we have $\displaystyle {\mathcal J}^\gamma(v)\geq -J(0,0)$ and, therefore, $\dis \displaystyle\inf_{v\in {{\mathcal U} }}{\mathcal J}^\gamma(v)$ exists. Let then $\{v_\gamma^n\}$ be a minimizing sequence such that $\dis d_\gamma=\lim_{n\to+\infty} {\mathcal J}^\gamma(v_\gamma^n)$. We have
\begin{align*}
	-J(0,0)\leq {\mathcal J}^\gamma(v_\gamma^n) =J(v_\gamma^n,0)-J(0,0) +\frac{\beta^2}{\gamma} \Big\|\xi(v_\gamma^n,0)(0)\Big\|_{G'}^2 \leq d_\gamma+1,
\end{align*}
for $n$ sufficiently large. From this, we deduce that there exists a positive constant $c_\gamma$, independent on $n$, such that the following estimates hold
\begin{equation}
	\Big\|v_\gamma^n\Big\|_{\mathcal{U}} \leq c_\gamma,\quad \Big\|y(v_\gamma^n,0)-\bar{z}\Big\|_{L^2(Q)}\leq 
c_\gamma,\quad\mbox{and}\;\; \frac{\beta}{\sqrt{\gamma}} \Big\|\xi(v_\gamma^n,0)(0)\Big\|_{G'} \leq c_\gamma.
\end{equation}

Then we deduce that the sequence $\{v_\gamma^n\}$ is bounded in $\mathcal{U}$  and therefore there exists a subsequence, converging weakly in $\mathcal{U}$ to $v_\gamma$. Moreover, thanks to the strict convexity of ${\mathcal J}^\gamma$, $v_\gamma$ is unique.
\end{proof}

\subsection{Characterization of the low-regret control}

This section is devoted to a characterization of the low-regret control fro \eqref{direct_pb_y} through the corresponding optimality system. To this end, we shall first prove the following result.

\begin{lemma}\label{lem:EL}
Let $v_\gamma$ be the optimal low-regret control solution of the minimization problem \eqref{infgamma}, and denote $y_\gamma:=y(v_\gamma,0)$ and $\xi_\gamma:=\xi(v_\gamma,0)$. Then, for all $w\in\mathcal U$ we have the inequality
\begin{equation}\label{euler_lagrange_part2}
	\beta\langle y_\gamma-\bar{z},y(w,0)\rangle_{L^2(Q)} + \mu\langle v_\gamma,w\rangle_{\mathcal{U}} + \frac{\beta^2}{\gamma}\left\langle \xi_\gamma(0),\xi(w,0)(0) \right\rangle_{G'} = 0.
\end{equation} 
\end{lemma}

\begin{proof}
We have
\begin{align*}
	\frac{{\mathcal J}^\gamma(v_\gamma+\lambda w) -{\mathcal J}^\gamma(v_\gamma)}{\lambda} =&\, J(v_\gamma+\lambda w,0) + \frac{\beta^2}{\gamma}\norm{\xi(v_\gamma+\lambda w,0)(0)}{G'}^2 - J(v_\gamma,0) - \frac{\beta^2}{\gamma}\norm{\xi(v_\gamma,0)(0)}{G'}^2
	\\
	=&\, {\beta}\norm{(y_\gamma-\bar{z})+\lambda y(w,0)}{L^2(Q)}^2 + \mu\norm{v_\gamma+\lambda w}{\mathcal{U}}^2 -{\beta}\norm{y_\gamma-\bar{z}}{L^2(Q)}^2 - \mu\norm{v_\gamma}{\mathcal{U}}^2 
	\\
	&+ \frac{\beta}{\gamma}\norm{\xi_\gamma(0) +\lambda \xi(w,0)(0)}{G'}^2 -\frac{\beta}{\gamma}\norm{\xi_\gamma(0)}{G'}^2 
	\\
	=&\, {\lambda^2\beta}\norm{y(w,0)}{L^2(Q)}^2 +2\lambda\beta\langle y_\gamma-\bar{z},y(w,0)\rangle_{L^2(Q)} + \mu\lambda^{2}\norm{w}{\mathcal{U}}^2 +2\mu\lambda\langle v_\gamma,w\rangle_{\mathcal{U}} 
	\\
	&+ \frac{\beta^2\lambda^2}{\gamma}\norm{\xi(w,0)(0)}{G'}^2 +\frac{2\beta^2\lambda}{\gamma}\langle \xi_\gamma(0),\xi(w,0)(0)\rangle_{G'}.
\end{align*}
Identity \eqref{euler_lagrange_part2} is then an immediate consequence of the Euler-Lagrange condition, which gives
\begin{align*}
	0&=\lim_{\lambda\rightarrow 0} \left(\frac{{\mathcal J}^\gamma(v_\gamma+\lambda w) -{\mathcal J}^\gamma(v_\gamma)}{\lambda}\right) 
	\\
	&={\beta} \langle y_\gamma-\bar{z},y(w,0)\rangle_{L^2(Q)} + N\langle v_\gamma,w\rangle_{\mathcal{U}} +\frac{\beta^2}{\gamma} \left\langle \xi_\gamma(0),\xi(w,0)(0) \right\rangle_{G'}, \quad \mbox{ for all } w \in\mathcal{U}.
\end{align*}
\end{proof}

\noindent By means of Lemma \ref{lem:EL}, we can now obtain the following characterization of the low-regret control.

\begin{proposition}\label{prop_characterization_part2}
The low-regret control $v_\gamma$ is characterized by the quadruplet 
\begin{align*}
	(v_\gamma,y_\gamma,\xi_\gamma,\sigma_\gamma) \in \mathcal{U} \times L^2(Q)\times L^2(Q)\times L^2(Q),
\end{align*}
unique solution to the optimality system
\begin{align}\label{sos_part2}
	\begin{cases}
		\displaystyle \frac{\partial y_\gamma}{\partial t} -\Delta y_\gamma+ \int_\Omega K(x,\theta,t)\,y_\gamma(\theta,t)\,d\theta = v_\gamma{\bf 1}_{\mathcal{O}}, & (x,t)\in Q,
		\\[10pt]
 		\displaystyle -\frac{\partial \xi_\gamma}{\partial t} -\Delta \xi_\gamma+ \int_\Omega K(\theta,x,t)\,\xi_\gamma(\theta,t)\,d\theta = y_\gamma, & (x,t)\in Q,
 		\\[10pt]
		\displaystyle\frac{\partial \sigma_\gamma}{\partial t} -\Delta \sigma_\gamma+ \int_\Omega K(x,\theta,t)\,\sigma_\gamma(\theta,t)\,d\theta = 0, & (x,t)\in Q,
		\\[10pt] 
		\displaystyle -\frac{\partial p_\gamma}{\partial t} -\Delta p_\gamma+ \int_\Omega K(\theta,x,t)\, p_\gamma(\theta,t)\,d\theta = {\beta}(y_\gamma-\bar{z})-\sigma_\gamma, & (x,t)\in Q,
		\\[10pt]
		y_\gamma=\xi_\gamma=\sigma_\gamma =p_\gamma = 0,  & (x,t)\in \Sigma,
		\\[10pt]
		\displaystyle y_\gamma(0)=0,\; \sigma_\gamma(0)=-\frac{\beta^2}{\gamma}\xi_\gamma(0),\; \xi_\gamma(T) =p_\gamma(T) = 0, & x\in\Omega,
	\end{cases}
\end{align}
where $\sigma_\gamma:=\sigma(v_\gamma,0)$ and $p_\gamma:=p(v_\gamma,0)$, and with the adjoint state equation
\begin{equation}\label{adjoint_state_eq_part2}
	 p_\gamma+\mu v_\gamma=0, \qquad \mbox{in} \quad \mathcal{O}\times ]0,T[.
\end{equation}
\end{proposition}

\begin{proof}
We introduce $\sigma_\gamma=\sigma(v_\gamma,0)$ unique solution to the problem
\begin{equation}\label{direct_pb_sigma}
	\begin{cases}
		\displaystyle\frac{\partial \sigma_\gamma}{\partial t} -\Delta \sigma_\gamma+\int_\Omega K(x,\theta,t)\,\sigma_\gamma(\theta,t)\,d\theta = 0, & (x,t)\in Q,
		\\
		\sigma_\gamma = 0, & (x,t)\in \Sigma,
		\\
		\sigma_\gamma(0,x) = -\displaystyle \frac{\beta}{\gamma}\xi_\gamma(0), & x\in\Omega.
	\end{cases}
\end{equation}
Then we multiply (\ref{direct_pb_sigma}) by $\xi(w,0)$ and we integrate by parts over $Q$, thus obtaining
\begin{align*}
	0&= \int_Q\left(-\frac{\partial \xi}{\partial t}-\Delta \xi+ \int_\Omega K(\theta,x,t)\,\xi(\theta,t)\,d\theta\right) \sigma_\gamma\,dxdt - \int_\Omega \xi(w,0)(0)\,\sigma_\gamma(0)\,dx
	\\
	&=\int_Q y(w,0)\,\sigma_\gamma\, dxdt - \int_\Omega \xi(w,0)(0)\,\sigma_\gamma(0)\,dx.
\end{align*}
Hence, 
\begin{align*}
	\left\langle \frac{\beta^2}{\gamma}\xi_\gamma(0),\xi(w,0)(0)
	\right\rangle_{G'} = 	 \langle\sigma_\gamma,y(w,0) \rangle_{L^2(Q)}.
\end{align*}
Then (\ref{euler_lagrange_part2}) reduces to
\begin{equation}\label{simplification_euler_part2}
	\langle {\beta}(y_\gamma-\bar{z})-\sigma_\gamma, y(w,0)\rangle_{L^2(Q)} + \mu\langle v_\gamma,w\rangle_{\mathcal{U}} \geq 0, \quad \mbox{ for all } w \in\mathcal{U}.
\end{equation}
Finally, we define the adjoint state $p_\gamma:=(x,t,v_\gamma)$ as the unique solution of:
\begin{equation}\label{adjoint_pb_p_gamma}
	\begin{cases}
		\displaystyle -\frac{\partial p_\gamma}{\partial t}-\Delta p_\gamma+\int_\Omega K(\theta,x,t)\,p_\gamma(\theta,t)\,d\theta = \beta(y_\gamma-\bar{z})-\sigma_\gamma, & (x,t)\in Q,
		\\
		p_\gamma = 0,  & (x,t)\in \Sigma,
		\\
		p_\gamma(x,T) = 0, & x\in\Omega.
	\end{cases}
\end{equation}
We multiply (\ref{adjoint_pb_p_gamma}) by $y(w,0)$ and we integrate by parts over $Q$ to obtain
\begin{align*}
	\displaystyle\int_Q ({\beta}(y_\gamma-\bar{z})-\sigma_\gamma)y(w,0)\,dxdt =\displaystyle\int_0^T\!\!\int_{\mathcal{O}} p_\gamma\,w\, dxdt.
\end{align*}
Then (\ref{simplification_euler_part2}) reads $\langle\, p_\gamma+\mu v_\gamma,\, w\rangle_{\mathcal{U}} = 0$, for all $w\in \mathcal{U}$,
since $\mathcal{U}$ is a vector space.
\end{proof}

\section{Numerical experiments}\label{sec:5}

In this section, we provide some numerical simulations and comment on the practical implementation of the optimal control and the low-regret control problems for non-local parabolic systems. In what follows, for all $a<b$ we shall use the notation $\inter{a,b}=[a,b]\cap \mathbb N$. 

\subsection{Discretization of the non-local model and numerical implementation of the optimal control problem}

Let $\Omega=(0,1)$ and consider the 1-d model given by

\begin{align}\label{direct_pb_y_simp}
	\begin{cases}
		\dis\frac{\partial y}{\partial t}- \frac{\partial^2 y}{\partial x^2}+ \int_{0}^{1} K(x,\theta)\,y(\theta,t)\,d\theta = v{\bf 1}_{\mathcal{O}}, & (x,t)\in Q,
		\\
		y = 0, &(x,t)\in \Sigma,
		\\
		y(0,x) = y_0(x), & x\in\Omega.
	\end{cases}
\end{align}  
where $K\in L^\infty(\Omega\times\Omega)$. For convenience, we shall consider the case when $K$ can be written in separated variables, i.e., $K(x,\theta)=K_1(x)K_2(\theta)$.

For the numerical tests, system \eqref{direct_pb_y_simp} (and its adjoint) are discretized in time by using a standard implicit Euler scheme and discretized in space by a finite-difference scheme adapted to the non-local term. For given $N,M\in\mathbb N^*$, we set $\delta t=T/M$ and $h=1/(N+1)$ and consider the following discretization for the space and time variables
\begin{align*}
	&0=x_0<x_1<\ldots< x_{N}<x_{N+1}=1, 
	\\
	&0=t_0<t_1<\ldots< T_{M}=T,
\end{align*}
where $x_i=i h$, $i\in\inter{0,N+1}$, and $t_n=n\delta t$, $n\in\inter{0,M}$. The numerical approximation of a function $f=f(x,t)$ at a grid point $(x_i,t_n)$ will be denoted as $f_i^n:=f(t_n,x_i)$ and, for fixed $n$, we write 
\begin{align*}
	f^n=\left(\begin{array}{ccc}f_1^n,  \dots, f_N^n\end{array}\right)^\top
\end{align*}
for the evaluation at the interior points. With this notation, we shall consider fully discrete systems of the form
\begin{equation}\label{eq:discr_nonlocal}
	\begin{cases}
		\displaystyle \frac{y^{n+1}-y^n}{\delta t}+\mathcal A_{h} \, y^{n+1}=\mathcal B_h v^{n+1}, & n \in \inter{0, M-1}, 
		\\ 
		y_0^{n+1}=y_{N+1}^{n+1}=0, & n\in\inter{0,M-1},	
		\\	
		y^0=y^0_h,
	\end{cases}
\end{equation}
where $y=(y^n)_{n\in\inter{0,M}}$ and $v=(v^n)_{n\in\inter{1,M}}$ are the (discrete) state and control variables, $y^{0}_{h}\in \mathbb R^{N}$ is the projection of the given initial datum $y(\cdot,0)$ on the space-mesh, $\mathcal A_{h}\in \mathbb R^{N\times N}$ is a suitable approximation of the differential and non-local operators, and $\mathcal B_{h}\in\mathbb R^{N\times N}$ stands for the corresponding approximation of the control operator. 
We will compute the matrix $\mathcal A_h\in \RR^{N\times N}$ as the sum of two parts, more precisely, $\mathcal A_h=\mathcal A_{h,D}+\mathcal A_{h,NL}$, the first one taking into account the second order differential operator and the second one the integral kernel. 

Using a standard finite-difference method, we construct the matrix $\mathcal A_{h,D}\in\mathbb R^{N\times N}$ as the usual tridiagonal matrix coming from the discretization of the Laplacian operator $-\partial^2_{xx}$ with homogeneous Dirichlet boundary conditions, that is, 
\begin{align*}
	(\mathcal A_h y)_{i}=-\frac{1}{h^2}(y_{i+1}-2y_{i} +y_{i-1}), \quad i=\inter{1,N}. 	
\end{align*}

To incorporate the effects of the integral kernel, we see that for a grid point $(x_j,t_n)$ with $j\in\inter{1,N}$ and $n\in\inter{1,M}$,
\begin{equation}\label{eq:kernel_discrete}
	\int_{0}^{1}K(x_j,\theta)y(\theta,t_n)d\theta \approx K_{1,j} \sum_{i=1}^{N} h K_{2,i} \, y_{i}^{n},
\end{equation}
and thus, writing in vector form, the action of the non-local term can be incorporated by means of the $N\times N$ matrix given by
\begin{equation*}
	\mathcal A_{h,NL}= h \left(\begin{array}{cccc}K_{1,1} &  &  &  \\  & K_{1,2} &  &  \\  &  & \ddots &  \\  &  &  & K_{1,N} \end{array}\right)\left(\begin{array}{cccc} K_{2,1} & K_{2,2} &\dots & K_{2,N} \end{array}\right)\otimes \left(\begin{array}{c}1 \\ 1 \\ \vdots \\ 1\end{array}\right) 
\end{equation*}
where $\otimes$ stands for the usual Kronecker product in $\RR^N$. Obviously, there are many other ways to discretize the integral \eqref{eq:kernel_discrete} but this is enough for our purposes.  Lastly, a natural choice for the control operator $\mathcal B_h$ is the $N\times N$ diagonal matrix with entries 
\begin{equation*}
	(\mathcal B_h)_{i,i}= \begin{cases}
		1 &\text{if } x_i\in\omega, \\ 0 &\text{if } x_i\notin\omega. 
\end{cases}
\end{equation*}
For the implementation of the optimal control problem, we consider the discrete functional
\begin{equation}\label{eq:discr_func}
	J_{h,\delta t}(v)= \beta \norm{y-\bar{z}}{L^2_{\delta t}(0,T;\RR^N)}^2+\norm{\mathcal B_h v}{L^2_{\delta t}(0,T;\RR^N)}^2, \quad \mathcal \beta>0,
\end{equation}
where $y=(y^n)_{n\in\inter{1,M}}$ is the solution to \eqref{eq:discr_nonlocal} and $\bar z=(\bar z^n)_{n\in\inter{1,M}}$ is a discretization of the target $\bar z\in L^2(\Omega\times(0,T))$. Without loss of generality, we have assumed that $\mu = 1$. In \eqref{eq:discr_func}, $L^2_{\delta t}(0,T;\RR^N)$ denotes the discretization of the space $L^2(\Omega\times(0,T))$, more precisely,
\begin{equation*}
	L^2_{\delta t}(0,T;\mathbb R^N):=\Big\{f=(f^n)_{n\in\inter{1,M}}, \ f^n\in\mathbb R^N, \ n\in\inter{1,M}\Big\},
\end{equation*}
endowed with the norm
\begin{align*} 
	\norm{f}{L^2_{\delta t}(0,T;\RR^N)}:= \left(\sum_{n=1}^{M}\delta t |f^n|^2\right)^{1/2},
\end{align*}
where $|\cdot|$ stands for the usual Euclidean norm in $\RR^N$. The natural associated inner product will be defined as 
\begin{align*}
	(f,g)_{L^2_{\delta t}(0,T;\mathbb R^N)}:= \sum_{n=1}^{M}\delta t (f^n,g^n),
\end{align*}
where $(\cdot,\cdot)$ is the usual dot product in $\RR^N$. For shortness, sometimes we simply write $\|\cdot\|_{L^2}$ instead of $\|\cdot\|_{L^2_{\delta t}(0,T;\RR^N)}$.

Once we have written \eqref{eq:discr_func} in such form, the obtention of an optimal control is quite standard. For completeness, we sketch it briefly. A straightforward computation yields that
\begin{equation}\label{eq:grad_J_h_delta}
	\nabla J_{h,\delta t}(v)=2 \beta S_{h,\delta t}^\star(S_{h,\delta t} v - \bar{w}) +2  \mathcal B_{h}^* \mathcal B_h v .
\end{equation}
In \eqref{eq:grad_J_h_delta}, $B_h^*$ denotes the matrix transpose of $B_h$, the operator $S_{h,\delta t}$ is defined as
\begin{align*}
	S_{h,\delta t}: L^2_{\delta t}(0,T;\mathbb R^N) &\to L^2_{\delta t}(0,T;\mathbb R^N) , \quad S_{h,\delta t}v:=y,
\end{align*}
where $y$ is the solution to the forward system 
\begin{equation}\label{eq:forward_GD}
	\begin{cases}
		\displaystyle \frac{y^{n+1}-y^n}{\delta t}+\mathcal A_{h} \, y^{n+1}=\mathcal B_h v^{n+1}, & n \in \inter{0, M-1}, 
		\\
		y^0=0,
	\end{cases}
\end{equation}
while the adjoint operator $S_{h,\delta t}^\star$ is defined as
\begin{align*}
	S_{h,\delta t}^\star: L^2_{\delta t}(0,T;\mathbb R^N) &\to L^2_{\delta t}(0,T;\mathbb R^N), \quad S_{h,\delta t}^\star z := p,
\end{align*}
where $p$ can be found from the solution to the backward system

\begin{equation}\label{eq:backward_GD}
	\begin{cases}
		\displaystyle \frac{p^{n}-p^{n+1}}{\delta t}+\mathcal A_h^* p^{n}= z^n, \quad n\in\inter{1,M}, 
		\\
		^{M+1}=0,
	\end{cases}
\end{equation}
and, finally, $\bar{w}:= \bar{z}-\mathring{y}$ where $\mathring{y}$ denotes the solution to \eqref{eq:discr_nonlocal} with $v\equiv 0$. In \eqref{eq:forward_GD}-\eqref{eq:backward_GD} (and similar formulas below), we shall omit the boundary conditions since we always consider homogeneous Dirichlet ones.

In this way, the optimal control $v=(v^n)_{n\in\inter{1,M}}$ can be readily found by solving the linear problem
\begin{equation}\label{eq:opt_control}
	(\beta S_{h,\delta t}^\star S_{h,\delta t}+ B_h^* B_{h})v=\beta S_{h,\delta t}^\star \bar{w}.
\end{equation}
where $S_{h,\delta t}^\star S_{h,\delta t} v$ corresponds to the evaluation of the cascade forward-backward system \eqref{eq:forward_GD}-\eqref{eq:backward_GD}. 

\subsubsection{Numerical results for the optimal control problem}

Let us set $T=1$ and consider the following parameters of the system
\begin{align}\label{eq:kernel_num}
	y_0(x)=2\sin(\pi x), \quad K_1(x)=\sin(5\pi x), \quad K_2(\theta)=20\times \mathbf{1}_{(0,0.5)}(\theta)\sin(\pi \theta) 
\end{align}

Since nothing changes in our theoretical analysis or the numerical implementation, in the remainder of this document we shall always consider a small diffusion parameter $\nu=0.1$ multiplying the Laplace operator. In Figure \ref{fig:soly_libre}, we plot the free solution of system \eqref{direct_pb_y_simp}, that is, the solution with $v\equiv 0$. We clearly observe the effect of the non-local kernel $K(x,\theta)$ in the resolution of the equation. Here, we have used $M=100$ and $N=60$ for the number of points in the discrete grid. 

\begin{figure}[ht!]
\centering
\includegraphics{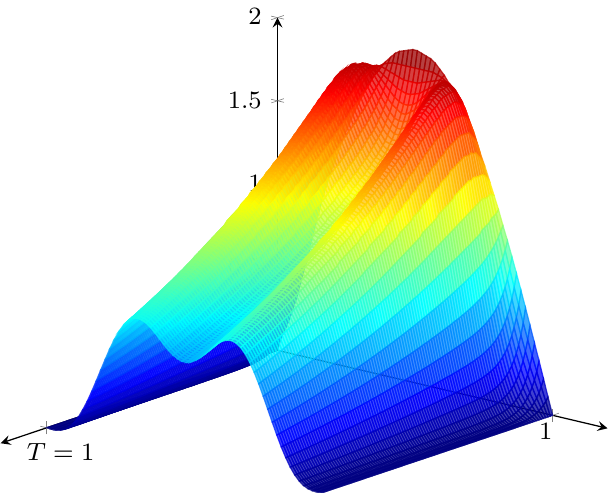}
\caption{Evolution in time of the uncontrolled solution.}\label{fig:soly_libre}
\end{figure}

Now, let us consider $\mathcal O=(0.2,0.8)$ and set the time-independent target function $ \bar{z}(x)=\sin(2\pi x)$. In Figure \ref{fig:soly_control_beta}, we plot the solution using the control obtained by solving the linear problem \eqref{eq:opt_control} for different values of $\mathcal \beta$. To solve this problem, we have used a standard Gradient Descent method. As usual in this type of problems, we observe that by increasing the parameter $\beta$, the approximation to the target $\bar{z}$ (i.e., $\|y(v)-\bar{z}\|_{L^2}$) is overall better, but the control cost increases. 

\begin{figure}[ht!]
\centering
\subfloat[$\beta=10$. Cost of control $\|B_h v\|_{L^2_{\delta t}(0,T;\mathbb R^{N})}=1.57137$]
{
\includegraphics{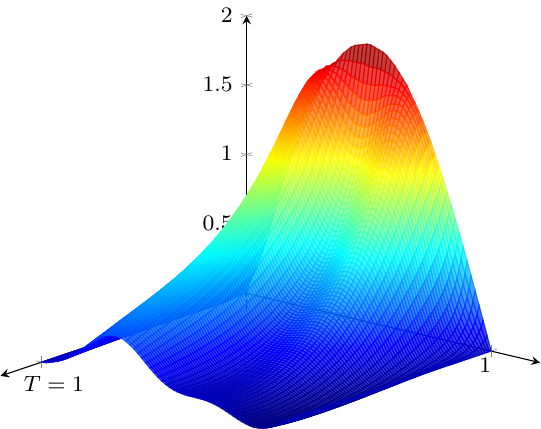}
}\qquad 
\subfloat[$\beta=10^2$. Cost of control $\|B_h v\|_{L^2_{\delta t}(0,T;\mathbb R^{N})}=4.19168 $]
{
\includegraphics{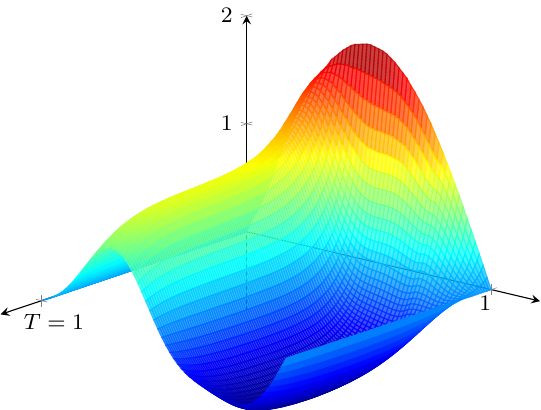}
} \\
\subfloat[$\beta=10^3$. Cost of control $\|B_h v\|_{L^2_{\delta t}(0,T;\mathbb R^{N})}=7.41933 $]
{
\includegraphics{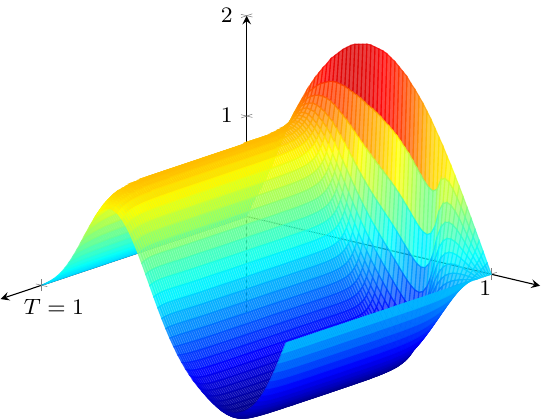}
}
\caption{Evolution in time of the controlled solution for different values of the penalization parameter $\beta$.}
\label{fig:soly_control_beta}
\end{figure}

As for the uncontrolled solution displayed in Figure \ref{fig:soly_libre}, we can clearly observe the effects of the non-local kernel. Indeed, even though the initial condition $y_0$, the target $\bar z$ and the control region $\mathcal O$ are somehow symmetric with respect to the point $x=0.5$, the distance of the solution $y$ to the target $z$ is not symmetric with respect to the point $x=0.5$ as it can be seen in Figure \ref{fig:steady_state}.

\begin{figure}[ht!]
\centering
\subfloat[Non-local]{
\begin{tikzpicture}[scale=0.75]

	\begin{axis}[xlabel={$x$},xmin=0,xmax=1,legend pos=outer north east,
	legend plot pos=left,
	legend style={cells={anchor=west},draw=none}]
	    
	\pgfplotstableread{figures_final/target_60_07-Jun-2021_13h11.org}\nonlocal    
	  
	\addplot[very thick, solid, red] table[x=x,y=yT2] \nonlocal;   
	\addplot[very thick,dashed] table[x=x,y=yT1] \nonlocal; 

	\end{axis}
	\end{tikzpicture}\label{fig:bad_data1}
}
\subfloat[Local]{
\begin{tikzpicture}[scale=0.75]

	\begin{axis}[xlabel={$x$},xmin=0,xmax=1,legend pos=outer north east,
	legend plot pos=left,
	legend style={cells={anchor=west},draw=none}]
    
	\pgfplotstableread{figures_final/target_60_07-Jun-2021_13h17.org}\local    
	  
	\addplot[very thick, solid, red] table[x=x,y=yT2] \local;   	
	\addplot[very thick,dotted] table[x=x,y=yT1] \local; \label{local-st}

	\end{axis}
\end{tikzpicture}\label{fig:bad_data2}
}
\caption{Comparison of the non-local steady state (dashed) and the local steady state (dotted) against the target function $\bar z$ (red solid).}
\label{fig:steady_state}
\end{figure}
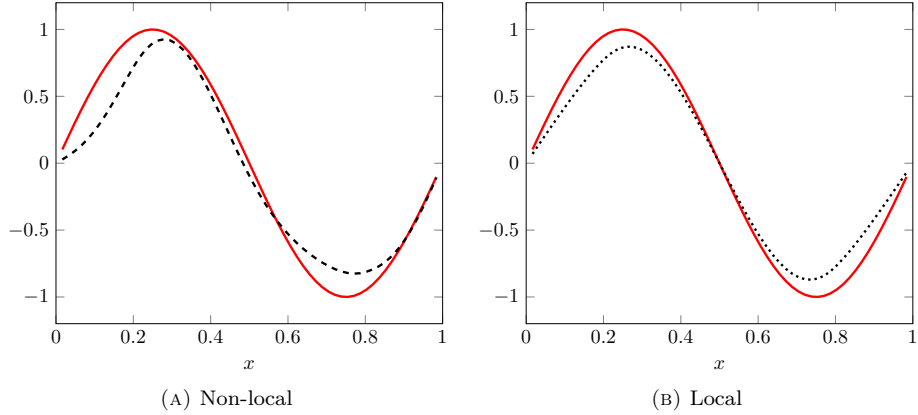

\subsection{Numerical implementation of the low-regret problem}

We turn our attention to the low regret problem. As introduced in Section \ref{sec:4} this problem aims to solve an optimal control problem in the case when the initial data is missing or incomplete. As proposed in the classical work by J.-L. Lions, we can address this problem by replacing the classical optimal control functional \eqref{min_J_part1} by the $\min$-$\max$ problem \eqref{min_J_part2}. Although it is feasible to solve \eqref{min_J_part2} (see e.g. \cite{DP18} and the reference within), here we will use problem \eqref{infgamma}, which transforms the original problem into a minimization one with two state variables. To this end, let us consider the fully-discrete version of the functional $J^\gamma(v)$ (see eq. \eqref{infgamma}) given by
\begin{align}\label{eq:func_discr_low}
	&J^\gamma_{h,\delta t}(v)= \beta \|y(v,0)-\bar z\|_{L^2_{\delta t}(0,T;\mathbb R^N)}^2+\|\mathcal B_h v\|^2_{L^2_{\delta t}(0,T;\mathbb R^N)}- \beta\|\bar z\|^2_{L^2_{\delta t}(0,T;\mathbb R^N)}+\frac{\beta^2}{\gamma}|\xi^{1}|^2 \notag 
	\\
	&v\in L^2_{\delta t}(0,T;\mathbb R^N)
\end{align}
where $y(v,0)=\left(y^n(v,0)\right)_{n\in\inter{1,M}}$ denotes the solution to \eqref{eq:discr_nonlocal} with control $v$ and initial datum $y^0\equiv 0$, $\bar z=(\bar z^n)_{n\in\inter{1,M}}$ is the discretization of the target $\bar z\in L^2(\Omega\times(0,T))$, and $\xi^1$ can be found from the backward equation
\begin{equation}\label{eq:backward_xi}
	\begin{cases}
		\displaystyle \frac{\xi^{n}-\xi^{n+1}}{\delta t}+\mathcal A_h^* \xi^{n}= y^n(v,0), \quad n\in\inter{1,M}, 
		\\
		\xi^{M+1}=0.
	\end{cases}
\end{equation}
A straightforward computation gives
\begin{align}\label{eq:deriv_J_gamma_discr}
	&\displaystyle \left(\nabla J^\gamma_{h,\delta t}(v),\hat{v}\right)= 2\beta \left(y(v,0)-\bar{z},y(\hat{v},0)\right)_{L^2_{\delta t}(0,T;\mathbb R^{N})}+2\left(\mathcal B_h v,\mathcal B_h \hat{v}\right)_{L^2_{\delta t}(0,T;\mathbb R^{N})}+2\beta\left(\sqrt{\tfrac{\beta}{\gamma}}\xi^1,\sqrt{\tfrac{\beta}{\gamma}}\hat{\xi}^1\right), \notag 
	\\ 
	&\mbox{ for all } \hat{v}\in L^2_{\delta t}(0,T;\RR^N),
\end{align}
where $\hat{\xi}^1$ comes from the sequence $\hat{\xi}=(\hat{\xi}^n)_{n\in\inter{1,M}}$, i.e., the solution to \eqref{eq:backward_xi} with right-hand side $y^n(\hat v,0)$.

Our goal is now to find a suitable expression for the gradient $\nabla J_{h,\delta t}^\gamma(v)$, since it will be used later in a gradient descent algorithm. We begin by defining the space $X_{h,\delta t}:= \mathbb R^N\times L^2_{\delta t}(0,T;\mathbb R^N)$, endowed with the canonical inner product $\langle\cdot,\cdot \rangle_{X_{h,\delta t}}:= (\cdot,\cdot)+(\cdot,\cdot)_{L^2_{\delta t}(0,T;\RR^N)}$. Consider the linear operator 
\begin{equation*}
	R_{h,\delta t}: L^2_{\delta t}(0,T;\mathbb R^N)\to X_{h,\delta t}, \quad R_{h,\delta t}v:=\left(\sqrt{\tfrac{\beta}{\gamma}}\xi^1,y\right)
\end{equation*}
where the pair $(\xi^1,y)$ can be found by solving the forward-backward system

\begin{equation}\label{eq:for-back_low}
	\begin{cases}
		\displaystyle \frac{y^{n+1}-y^n}{\delta t}+\mathcal A_{h} \, y^{n+1}=\mathcal B_h v^{n+1}, & n \in \inter{0, M-1}, 
		\\
		\displaystyle \frac{\xi^{n}-\xi^{n+1}}{\delta t}+\mathcal A_h^* \xi^{n}= y^n, & n\in\inter{1,M}, 
		\\
		y^0=0, \quad \xi^{M+1}=0
	\end{cases}
\end{equation}

We observe that system \eqref{eq:for-back_low} is in cascade form: given $v\in L^2_{\delta t}(0,T;\mathbb R^N)$, we can solve first for $y=(y^n)_{n\in\inter{1,M}}$ forward in time and then use this information to solve for $\xi=(\xi^n)_{n\in\inter{1,M}}$ backwardly to recover the datum $\xi^1$. 

Now, let us compute the adjoint operator $R_{h,\delta t}^\star$. We introduce the following system: for given $\sigma_0\in\mathbb R^N$ and $f \in L^2_{\delta t}(0,T;\mathbb R^N)$ we set
\begin{align}\label{eq:back-for_low}
	\begin{cases}
		\displaystyle \frac{\sigma^{n+1}-\sigma^n}{\delta t}+\mathcal A_{h} \, \sigma^{n+1}=0, & n \in \inter{0, M-1}, 
		\\
		\displaystyle \frac{p^{n}-p^{n+1}}{\delta t}+\mathcal A_h^* p^{n}= f^n-\sigma^n, & n\in\inter{1,M}, 
		\\
		\sigma^0=\displaystyle -\tfrac{\beta}{\gamma}\sigma_0, \quad p^{M+1}=0
	\end{cases}
\end{align}

As before, we note that \eqref{eq:back-for_low} is in cascade form, namely, given $\sigma_0 \in \mathbb R^N$ we can solve for $\sigma=(\sigma^n)_{n\in\inter{1,M}}$ and with this we can solve for $p=(p^n)_{n\in\inter{1,M}}$ for any given $f\in L^2_{\delta t}(0,T;\mathbb R^N)$.

Multiplying the first equation of \eqref{eq:for-back_low} by $p^{n+1}$ and summing over $n$, we have by a direct computation that
\begin{equation}\label{eq:duality_1}
	\left(y,f-\sigma \right)_{L^2_{\delta t}(0,T;\mathbb R^N)}=(v,p)_{L^2_{\delta t}(0,T;\RR^N)}
\end{equation}
where we have used that $y$ and $p$ have zero initial and final datum, respectively. Analogously, multiplying the first equation of \eqref{eq:back-for_low} by $\xi^{n+1}$, summing over $n$ and using the initial and final conditions, we get
\begin{equation}\label{eq:duality_2}
	(\sigma,y)_{L^2_{\delta t}(0,T;\RR^N)}=-\left(\tfrac{\beta}{\gamma}\sigma_0,\xi^1\right).
\end{equation}
Combining \eqref{eq:duality_1} and \eqref{eq:duality_2} and recalling the definition of the operator $R_{h,\delta t}$, we have
\begin{equation*}
	\left\langle R_{h,\delta t} v, \left(\sqrt{\tfrac{\beta}{\gamma}}\sigma_0,f\right) \right\rangle_{X_{h,\delta t}}=\left\langle \left(\sqrt{\tfrac{\beta}{\gamma}}\xi^1,y\right),\left(\sqrt{\tfrac{\beta}{\gamma}}\sigma_0,f\right) \right\rangle_{X_{h,\delta t}}=(v,p)_{L^2_{\delta t}(0,T;\RR^N)}
\end{equation*}
Thus, $R_{h,\delta t}^\star: X_{h,\delta t}\to L^2_{\delta t}(0,T;\RR^N)$ is defined as 
\begin{align*}
	R^\star_{h,\delta t}(\sqrt{\tfrac{\beta}{\gamma}}\sigma_0,f):= p,
\end{align*}
where $p=(p^n)_{n\in\inter{1,M}}$ can be found from the solution to the forward-backward system \eqref{eq:back-for_low}. By linearity and using the above definitions, we can then rewrite \eqref{eq:deriv_J_gamma_discr} as
\begin{align*}
	\left(\nabla J^\gamma_{h,\delta t}(v),\hat{v}\right)= 2\beta\left\langle R_{h,\delta t} v, R_{h,\delta t}\hat v \right\rangle_{X_{h,\delta t}}+ 2\left(\mathcal B_h v,\mathcal B_h \hat{v}\right)_{L^2_{\delta t}(0,T;\mathbb R^{N})}-2\beta(\bar{z},S_{h,\delta t}\hat v)_{L^2_{\delta t}(0,T;\mathbb R^{N})}
\end{align*}
whence $\nabla J^\gamma_{h,\delta t}(v)= 2 \beta R^\star_{h,\delta t} R_{h,\delta t} v+2\mathcal B_h^*\mathcal B_h v-2\beta S_{h,\delta t}^\star \bar z$, and thus the optimal control we are looking for can be computed by solving the linear problem
\begin{equation}\label{eq:lin_regret}
	\left(\beta R^\star_{h,\delta t} R_{h,\delta t}+\mathcal B_h^*\mathcal B_h\right) v=\beta S_{h,\delta t}^\star \bar z.
\end{equation}

Observe that the structure of the problem is the same as in \eqref{eq:opt_control}, but this time we have to evaluate the operator $R_{h,\delta t}^\star R_{h,\delta t}$ which amounts to solve four equations instead of two.

\subsubsection{Numerical results for the low-regret optimal control problem}

We fix $T=1$ and consider the kernel functions $K_i$, $i=1,2$ given in \eqref{eq:kernel_num}. As before, we consider the control set $\mathcal O=(0.2,0.8)$ and the time-independent target function $\bar z(x)=\sin(2\pi x)$. As for the optimal control problem, we will solve the linear problem \eqref{eq:lin_regret} by using a standard gradient descent method. 

In Figure \ref{fig:soly_control_regret}, we show the evolution in time of the state variable $y$ and the adjoint state $\xi$ controlled with the low-regret control $v^\gamma$ coming from the minimization of the functional \eqref{eq:func_discr_low}. In this case, we have tuned the parameters to $\beta=100$ and $\gamma=1$. We observe in Figure \ref{fig:soly_control_regret}a, the evolution of the state variable $y$ starting from 0 and transitioning into a steady state which seems to be similar to the steady state shown in Figure \ref{fig:soly_control_beta}a. In Figure \ref{fig:soly_control_regret}b, we observe the backward evolution of the adjoint variable $\xi$ starting from the zero data at time $T=1$.
%
%
%
%
%

\begin{figure}[ht!]
\centering
\subfloat[$\beta=10$. Cost of control $\|B_h v\|_{L^2_{\delta t}(0,T;\mathbb R^{N})}=1.57137 $]
{
\includegraphics{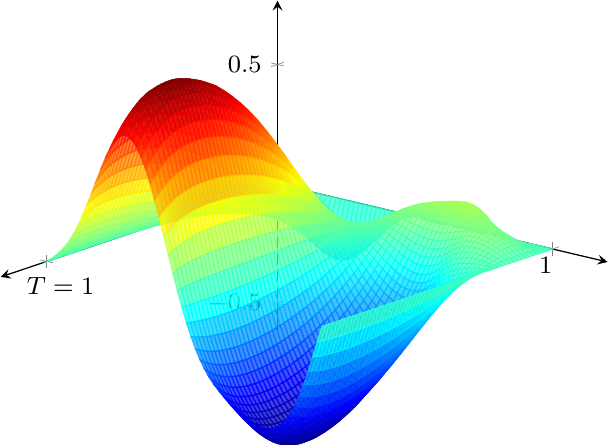}
} \qquad 
\subfloat[$\beta=10^2$. Cost of control $\|B_h v\|_{L^2_{\delta t}(0,T;\mathbb R^{N})}=4.19168 $]
{
\includegraphics{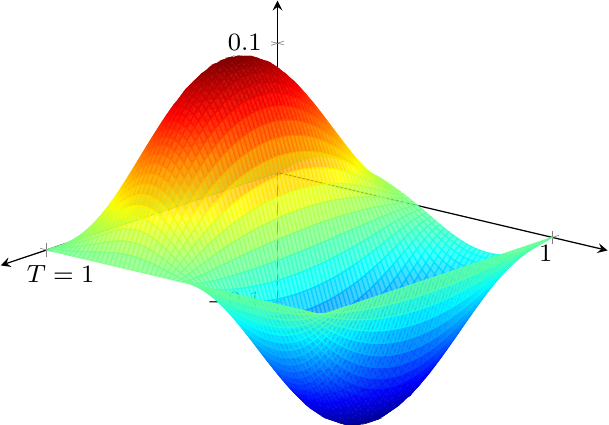}
}
\caption{Evolution in time of the controlled solution $y$ and the adjoint state $\xi$ with low-regret control with parameters $\beta=10$ and $\gamma=0.1$.}
\label{fig:soly_control_regret}
\end{figure}

In Table \ref{tab:low_reg_table}, we have collected data for the low-regret control problem with different values of $\beta$ and $\gamma$. From there, we can conclude that as in the optimal control problem, the higher the value of $\beta$ is, the better approximation to the target function $\bar{z}$ is.  We also see that the parameter $\gamma$ influences greatly the behavior of the control problem in the sense that for given $\beta>0$, lower values of $\gamma$ translate into a smaller $L^2$-norms of the control, affecting the overall approximation of the target $\bar z$. 

\begin{table}
\centering
\renewcommand{\arraystretch}{1.35}
\begin{tabular}{ |c||c|c|c|  }
\hline
\multicolumn{4}{|c|}{$\beta=1$} \\
\hline
$\gamma$ & $J_{h,\delta t}^\gamma(v)$ & $\|B_h v\|_{L^2_{\delta t}(0,T;\RR^N)}$ & $\|y(v,0)-\bar z\|_{L^2_{\delta t}(0,T;\RR^N)}$ \\
\hline
10 & -0.0132321 & 0.11325 & 0.688434 \\ 1 & -0.0132209 & 0.113154 & 0.68845 \\ 0.1 &-0.0131106 & 0.112213 & 0.688605 \\ 0.01 & -0.0121515 & 0.10422 & 0.689959 \\
 \hline 
  \multicolumn{4}{|c|}{$\beta=10$} \\
  \hline
10   & -1.03067    & 0.882627  &   0.564147  \\
 1 &   -0.968289  & 0.829583   &0.572797 \\
 0.1 & -0.682832   & 0.631448   & 0.61298   \\
  0.01 &-0.43599   & 0.558128  & 0.648367   \\
  \hline
    \multicolumn{4}{|c|}{$\beta=10^2$} \\
  \hline
   10   & -27.1986    & 2.37894  &   0.360065  \\
  1 &  -17.9241 & 2.10709 & 0.500576  \\
  0.1 &-14.9387  &  2.15546 & 0.546241   \\
 \hline
\end{tabular}
\caption{Optimal energy $J^\gamma(v)$, cost of control $\|B_h v\|_{L^2}$ and distance to the target $\|y(v,0)-\bar z\|_{L^2}$ for the low-regret control problem with different values of $\gamma$ and $\beta$}
\label{tab:low_reg_table}
\end{table}

We conclude the discussion by showing that the low-regret control can be used for controlling equations in the case of missing data. To this end, we take the simulation parameters $\beta=100$, $\gamma=1$, $T=1$, and the kernels $K_1$ and $K_2$ of \eqref{eq:kernel_num}. Using our computational tool, we can compute the low-regret control $v^\gamma$ and use it for controlling system \eqref{eq:discr_nonlocal} with different initial conditions. In Table \ref{tab:comp_1}, we have collected some information about the performance of the low-regret control $v^\gamma$ against the \textit{real} optimal control $v^{opt}$ (computed with $\beta=100$) and the uncontrolled case for different initial data. 

We can see that the low-regret control $v^\gamma$ does not make things worse as compared to the uncontrolled case, which is consistent with the goal of this strategy, but it is not as good as the performance achieved with the optimal control computed with the full knowledge of the initial datum. Nonetheless, the low-regret control has the advantage that is has to be computed only once and then it can be used to control the system for a wide variety of initial data. We shall remark that there is still room for improvement as shown in Table \ref{tab:comp_2}. There, we have increased to $\gamma=10$ and computed the corresponding low regret control to obtain a lower optimal energy and smaller distance of the target (as compared to the last two columns of Table \ref{tab:comp_1}). Nevertheless, how to choose effectively this parameter depends largely on the application and the experiment performed.

\begin{table}
\centering
\renewcommand{\arraystretch}{1.35}
\begin{tabular} {|c|c|c|c|c|c|c|}
\cline{2-7} 
\multicolumn{1}{c|}{}  & \multicolumn{2}{c|}{Uncontrolled} & \multicolumn{2}{c|}{Optimal control} & \multicolumn{2}{c|}{Low-regret control} \\
\hline
Initial datum & $J(0)$ & $\|y(0)-\bar z\|_{L^2}$ & $J(v^{opt})$ & $\|y(v^{opt})-\bar z\|_{L^2}$ & $J(v^\gamma)$ & $\|y(v^{\gamma})-\bar z\|_{L^2}$ \\
\hline 
$\sin^{10}(\pi x)$ & 27.62901 & 0.74336 & 8.82497 & 0.28957 & 17.43680 & 0.55176  \\
3 & 191.09119 & 1.95495 & 45.17162 & 0.70582 & 180.79790 & 1.88988 \\
$\mathbf 1_{(0.5,0.8)}(x)-\mathbf 1_{(0.2,0.5)}(x)$ & 37.50553 & 0.86609 & 13.85573 & 0.38048 & 26.43313 & 0.69597  \\
$\sin(\tfrac{1}{3}\pi x)+0.3\cos(\tfrac{15}{4}\pi x)$ & 32.88063 &0.81093 &10.83359 &0.32834 &22.34038 &0.63444 \\
\hline
\end{tabular}
\caption{Comparison of the optimal energy $J(\cdot)$ and the distance to the target for the controlled solution $\|y(\cdot)-\bar{z}\|_{L^2}$ for system \eqref{eq:discr_nonlocal} controlled with different control functions and initial conditions.}
\label{tab:comp_1}
\end{table}

\begin{table}
\centering
\renewcommand{\arraystretch}{1.35}
\begin{tabular} {|c|c|c|}
\hline
Initial datum & $J(v^\gamma)$ & $\|y(v^{\gamma})-\bar z\|_{L^2}$ \\
\hline 
$\sin^{10}(\pi x)$ &11.92807 &0.42667  \\
3  &175.15713 &1.85651  \\
$\mathbf 1_{(0.5,0.8)}(x)-\mathbf 1_{(0.2,0.5)}(x)$  &17.98742 &0.55067   \\
$\sin(\tfrac{1}{3}\pi x)+0.3\cos(\tfrac{15}{4}\pi x)$ &15.81261 &0.50965  \\
\hline
\end{tabular}

\caption{Effect of the parameter $\gamma$ on the controlled solution with low-regret control.}
\label{tab:comp_2}
\end{table}

\bibliographystyle{siam}
\bibliography{bib_opt_nonlocal}

\end{document}